\documentclass[11 pt]{amsart}
\usepackage{amscd,amssymb,amsthm}
\usepackage[arrow,matrix]{xy}
\usepackage{graphicx}
\newtheorem{thm}[subsection]{Theorem}
\newtheorem{lem}[subsection]{Lemma}

\newtheorem{cor}[subsection]{Corollary}

\theoremstyle{definition}
\newtheorem{definition}[subsection]{\textmd{Definition}}

\numberwithin{equation}{section} 

\begin{document}
\title[some inequalities on the norms of special matrices with particular sequences]{Some inequalities on the norms of special matrices with generalized Tribonacci and generalized Pell Padovan sequences }
\author[Z. Raza, M. Riaz and M. A. Ali]{Zahid Raza, Muhammad Riaz, and Muhammad Asim Ali}
\address{Department of Mathematics, College of Sciences\\University of Sharjah, Sharjah, UAE.}
         \email{zraza@sharjah.ac.ae}
         \address{Department of Mathematics, University of The Punjab, Lahore,
         Pakistan.}
         \email{mriaz.math@pu.edu.pk}
 \address{Department of Sciences and Humanities, National University of Computer and Emerging Sciences, Lahore,
         Pakistan.}
\email{masimali99@gmail.com}

%\subjclass{?????}

\keywords{Circulant,$r-$Circulant,semi-circulant, Hankel, spectral
norm, Euclidean norm.}
\begin{abstract} In this paper some properties of generalized tribonacci and generalized Padovan sequence are presented.
  Also the Euclidean norms of circulant, $r$-circulant, semi-circulant and Hankle matrices with above mentioned sequences
  are calculated. The upper and lower bounds of spectral norms are also obtained.
\end{abstract}
\maketitle
\section{Introduction}
Circulant matrices have many application in solving the different types of ordinary and partial differential equations.
 In computations mathematics these matrices have a key role in signal processing, digital image etc.\\

There are many results about some special matrices and their norms . Solak \cite{Solak} have found the bounds of spetral norms
of circulant matrices with Fibonacci and lucas numbers.
Shen and Cen \cite{Shen} established lower and upper bounds for the spetral norms of $r$-circulant matrices with
Fibonacci and Lucas numbers. Akbulak and Bozkurt \cite{Akbulak}
found upper and lower bounds for the spetral norms of Toeplitz matrices. In \cite{Kocer} Kocer obtained
 Circulant, Negacyclic and Semi-circulant matrices with the modified Pell, Jacobsthal and Jacobsthal- Lucas numbers.
  Halici \cite{Halici} investigated some inequality and norms of hankel matrices involving Pell,Pell Lucas numbers.
 In \cite{Arzu} establish some properties of
 circulant matrix with third order recurrence relation. In \cite{Yazlik Y} they defined spetral norm,
 eigenvalues and determinant of circulant matrix involving the generalized $k$ Horadam numbers. \\
In view of above study, we obtained the norms of different types of
matrices like circulant, $r$-circulant, semi-circulant and Hankel
matrices with generalized tribonacci and generalized pell padovan
sequence. Furthermore, Euclidean norms, and maximum row and maximum
column norms are obtained.
 Upper and lower bounds of spetral norm are also calculated in this paper.\\
 The third order recurrence relation is defined as:
\begin{equation}\label{e1}
{{Q}_{n}}=p{{Q}_{n-1}}+{q{Q}_{n-2}+r{Q}_{n-3}},
\end{equation}
with initial conditions ${{Q}_{0}}=a$, ${{Q}_{1}}=b$, ${{Q}_{2}}=c$ where $a$, $b$  and
$c$ are positive integer.\\
If we take $p=q=r=1$ , then generalized tribonacci sequence is given as :
\begin{equation}\label{e2}
{{R}_{n}}={{R}_{n-1}}+{{R}_{n-2}+{R}_{n-3}}.
\end{equation}
with intial conditions ${{R}_{1}}=b$ ,${{R}_{1}}=b$, ${{R}_{2}}=c$.\\
Take $p=0$, $q=r=1$ in (\ref{e1}).\\
The tribonacci sequence is a generalization of Fibonacci sequence, which is defined as:
\begin{equation}
{{T}_{n}}={{T}_{n-1}}+{{T}_{n-2}+{T}_{n-3}}.
\end{equation}
with intial conditions ${{T}_{1}}=0$ ,${{T}_{1}}=1$, ${{T}_{2}}=1$.\\

The Generalized Pell Padovan sequence is given as :
\begin{equation}\label{e3}
{{Z}_{n}}={{Z}_{n-2}+{Z}_{n-3}},
\end{equation}
with intial conditions ${{Z}_{1}}=b$ ,${{Z}_{1}}=b$, ${{Z}_{2}}=c$.
\section{Preliminaries}
First we give some preliminaries about some special matrices.\\
A matrix ${{A}}={{A}}_{r}=\left( {{{a}}_{ij}}
\right)\in {{M}_{n,n}}(\mathbb{C})$ is called $r-$circulant on
any integer sequence $U_{n}$, if it is of the form
\begin{equation}
{{{{a}}}_{ij}}=\left\{ \begin{array}{*{35}{l}}
   {{{U}}_{j-i}} & \quad j\ge i  \\
   r{{{U}}_{n+j-i}} & \quad j<i\, ,  \\
\end{array} \right.
\end{equation}
where $r \in\mathbb{C}.$ If $r$=1, then matrix ${A}$ is
called circulant.\\
A circulant matrix $B\in {{M}_{n,n}}(\mathbb{C})$ denoted by  $CIRC(a_{1},a_{2},a_{3},\cdots ,a_{n})$ is a matrix of the form \\
$$B=\left(
     \begin{array}{ccccc}
       a_{1} & a_{2} & \cdot \cdot \cdot& a_{n-1} & a_{n} \\
       a_{n} & a_{1} & a_{2} & \cdot \cdot \cdot  & a_{n-1} \\
        \cdot& a_{n} & a_{1} & \cdot & \cdot\\
       a_{3} &\cdot & \cdot & \cdot & a_{2} \\
       a_{2} & a_{3} & \cdot \cdot \cdot & a_{n} & a_{1} \\
     \end{array}
   \right)
.$$\\
A left circulant matrix $B\in {{M}_{n,n}}(\mathbb{C})$ denoted by  $LCIRC(a_{1},a_{2},a_{3},\cdots ,a_{n})$ is a matrix of the form \\
$$C=\left(
     \begin{array}{ccccc}
       a_{1} & a_{2} & \cdot \cdot \cdot& a_{n-1} & a_{n} \\
       a_{2} & a_{1} & \cdot \cdot \cdot &  a_{n}  & a_{1} \\
        a_{3}&  &  & \cdot & \cdot\\
       &a_{n} & a_{1} & \cdot & a_{n-2} \\
       a_{n} & a_{1} & \cdot \cdot \cdot & a_{n-2} & a_{1} \\
     \end{array}
   \right)
.$$

A matrix ${A}=\left( {{{{a}}}_{ij}}\right)\in
{{M}_{n,n}}(\mathbb{C})$ is called semi-circulant on any integer sequence, if it is of the form\\
\begin{equation*}
{{{a}}_{ij}}=\left\{ \begin{array}{*{35}{l}}
   {{{U}}_{j-i+1}} & \quad i\le j  \\
   0 & \quad \mbox{otherwise}. \\
\end{array} \right.
\end{equation*}
A Hankel matrix on any integer sequence is defined as:\\
${H}=\left( {{{h}}_{ij}} \right)\in
{{M}_{n,n}}(\mathbb{C})$, where
${{{h}}_{ij}}={{{U}}_{i+j-1}}.$\\
Similarly, a matrix ${A}=\left( {{{a}}_{ij}}\right)\in
{{M}_{n,n}}(\mathbb{C})$ is Toeplitz matrix on any integer sequence, if it is of the form $a_{ij}=U_{i-j}$.\\
The ${{\ell}_{p}}$ norm of a matrix ${A}=({{a}}_{ij})\in {{M}_{n,n}}(\mathbb{C})$ is defined by\\
$${{\left\|{A} \right\|}_{p}}={{\left(
\sum\limits_{i=1}^{m}{\sum\limits_{j=1}^{n}{{{\left| {{{a}}_{ij}}
\right|}^{p}}}} \right)}^{{}^{1}/{}_{p}}},\,\,\,\left( 1\le p\le
\infty  \right).$$
If  $p=\infty $, then ${{\left\| {A} \right\|}_{\infty
}}=\underset{p\to \infty }{\mathop{\lim }}\,{{\left\| {A}
\right\|}_{p}}=\underset{i,j}{\mathop{\max }}\,\left|
{{{{a}}}_{ij}} \right|.$

The Euclidean (Frobenius) norm of the matrix ${A}$ is defined
as:
$${{\left\| {A} \right\|}_{E}}={{\left(
\sum\limits_{i=1}^{m}{\sum\limits_{j=1}^{n}{{{\left| {{{{a}}}_{ij}}
\right|}^{2}}}} \right)}^{{}^{1}/{}_{2}}}.$$ The spectral norm of
the matrix ${A}$ is given as: $${{\left\|
{A}\right\|}_{2}}=\sqrt{\underset{1\le i\le n}{\mathop{\max
}}\,\left| {{\beta }_{i}} \right|}\, ,$$ where ${{\beta }_{i}}$ are
the eigenvalues of the matrix ${{\left( {\bar{{A}}}
\right)}^{t}}{A}.$\\
The following inequality between Euclidean and spectral norm holds
\cite{Zielke}
\begin{equation}\label{e6}
\frac{1}{\sqrt{n}}{{\left\| {A} \right\|}_{E}}\le {{\left\|
{A} \right\|}_{2}}\le {{\left\| {A} \right\|}_{E}}
\end{equation}

\begin{definition}\cite{Reams}\label{d.2}
Let ${A}=({{{a}}_{ij}})$ and
${B}=({{{b}}_{ij}})$  be $m\times n$ matrices. Then,
the Hadamard product of ${A}$ and ${B}$ is given by
 $${A}\circ {B}=({{a}}_{ij}{{b}}_{ij}).$$
\end{definition}
\begin{definition}\cite{Solak}\label{d.3}
The maximum column length norm ${{c}_{1}}(.)$ and maximum row length
norm  ${{r}_{1}}(.)$ for $m\times n$ matrix ${A}
=({{{{a}}}_{ij}})$ is defined as ${{c}_{1}}({A})=
\sqrt{\underset{j}{\mathop{\max }}\,\sum\limits_{i}{{{\left|
{{{a}}_{ij}} \right|}^{2}}}}$ and ${{r}_{1}}({A})=
\sqrt{\underset{i}{\mathop{\max }}\,\sum\limits_{j}{{{\left|
{{{a}}_{ij}} \right|}^{2}}}}$ respectively.
\end{definition}
\begin{thm}\cite{Mathias}\label{t.1}
Let ${A}=({{{a}}_{ij}})$, $B=({{{b}}_{ij}})$
and ${C}=({{{{c}}}_{ij}})$ be $p\times q$ matrices. If
${C}={A}\circ {B}$, then ${{\left\| {C}
\right\|}_{2}}\le {{r}_{1}}({A}){{c}_{1}}({B}).$
\end{thm}

\section{Generalized tribonacci sequence}

First we need to describe the some properties of generalized tribonacci sequence.
\begin{lem}\label{l.1}
Let $Q_n$ be the $n$-th term of third order linear recurrence relation.Then,\\
\[\sum\limits_{i=1}^{n}{{{Q}_{n}}}=\frac{(1-p)({{Q}_{n+1}}+{{Q}_{n+2}}-{{Q}_{2}}-{{Q}_{1}})\,-(p{{Q}_{2}}+r{{Q}_{0}})+{{Q}_{n+3}}-q{{Q}_{n+1}}}{p+q+r-1}\]\emph{}
\end{lem}
\begin{proof}
let us take\\
 $${{Q}_{i}}=p{{Q}_{i-1}}+q{{Q}_{i-2}}+r{{Q}_{i-3}}$$
$$q{{Q}_{i-2}}+r{{Q}_{i-3}}={{Q}_{i}}-p{{Q}_{i-1}}$$
$i=3,4,5,\cdots, n+3$
$$q{{Q}_{1}}+r{{Q}_{0}}={{Q}_{3}}-p{{Q}_{2}}$$
$$q{{Q}_{2}}+r{{Q}_{1}}={{Q}_{4}}-p{{Q}_{3}}$$
$$q{{Q}_{3}}+r{{Q}_{2}}={{Q}_{5}}-p{{Q}_{4}}$$
$$\cdot\cdot\cdot\cdot\cdot\cdot\cdot\cdot\cdot=\cdot\cdot\cdot\cdot\cdot\cdot\cdot\cdot\cdot$$
$$\cdot\cdot\cdot\cdot\cdot\cdot\cdot\cdot\cdot=\cdot\cdot\cdot\cdot\cdot\cdot\cdot\cdot\cdot$$
$$q{{Q}_{n+1}}+r{{Q}_{n}}={{Q}_{n+3}}-p{{Q}_{n+2}}.$$
Adding all the terms vertically.\\\
$$q\sum\limits_{i=1}^{n+1}{{{Q}_{i}}}+r\sum\limits_{i=0}^{n}{{{Q}_{i}}}={{Q}_{n+3}}-p{{Q}_{2}}+(1-p)\sum\limits_{i=3}^{n+2}{{{Q}_{i}}}$$
$$q\sum\limits_{i=1}^{n+1}{{{Q}_{i}}}+r\sum\limits_{i=0}^{n}{{{Q}_{i}}}={{Q}_{n+3}}-p{{Q}_{2}}+(1-p)\sum\limits_{i=3}^{n+2}{{{Q}_{i}}}.$$
Let $A=\sum\limits_{i=1}^{n}{{{Q}_{i}}}$.
$$q(A+{{Q}_{n+1}})+r(A+{{Q}_{0}})={{Q}_{n+3}}-p{{Q}_{2}}+(1-p)(A+{{Q}_{n+1}}+{{Q}_{n+2}}-{{Q}_{2}}-{{Q}_{1}})$$
$$\sum\limits_{i=1}^{n}{{{Q}_{n}}}=\frac{(1-p)({{Q}_{n+1}}+{{Q}_{n+2}}-{{Q}_{2}}-{{Q}_{1}})\,-(p{{Q}_{2}}+r{{Q}_{0}})+{{Q}_{n+3}}-q{{Q}_{n+1}}}{p+q+r-1}.$$
\end{proof}
\begin{cor}
For any $n>1$. \\
$$\sum\limits_{i=1}^{n}{{{R}_{i}}=\frac{{{R}_{n+3}}-{{R}_{n+1}}-{{R_2}}-{{R_1}}}{2}}$$
\end{cor}\label{cr.1}
\begin{lem}\label{l.2}
For any $n>1$.
$$\sum\limits_{k=1}^{n}{R_{k}^{2}}=H_{n}=\frac{4{{R}_{n}}{{R}_{n+1}}-4{{R}_{0}}{{R}_{1}}-{{\left( {{R}_{n+1}}-{{R}_{n-1}} \right)}^{2}}+{{\left( {{R}_{-2}}+{{R}_{0}} \right)}^{2}}}{4}$$
$$\sum\limits_{k=1}^{n}{{{R}_{k}}}{{R}_{k-2}}=B_n=\frac{{{\left( {{R}_{n+1}}+{{R}_{n-1}} \right)}^{2}}-{{\left( {{R}_{-2}}+{{R}_{0}} \right)}^{2}}}{4}$$

\end{lem}
\begin{proof}
Take
$${{R}_{k}}{{R}_{k-1}}=\left( {{R}_{n-1}}+{{R}_{n-2}}+{{R}_{n-3}} \right){{R}_{k-1}}$$
$${{R}_{k}}{{R}_{k-1}}=R_{k-1}^{2}+{{R}_{k-2}}{{R}_{k-1}}+{{R}_{k-3}}{{R}_{k-1}}$$
$$\sum\limits_{k=1}^{n}{{{R}_{k}}{{R}_{k-1}}}=\sum\limits_{k=1}^{n}{R_{k-1}^{2}+}\sum\limits_{k=1}^{n}{{{R}_{k-2}}{{R}_{k-1}}+}\sum\limits_{k=1}^{n}{{{R}_{k-3}}{{R}_{k-1}}}.$$
$${{R}_{n}}{{R}_{n-1}}=\left( {H_{n}}+R_{0}^{2}-R_{n}^{2} \right)+{{R}_{0}}{{R}_{-1}}+\left( {B_{n}}+{{R}_{-2}}{{R}_{0}}-{{R}_{n}}{{R}_{n-2}} \right).$$
Where
 $${H_{n}}=\sum\limits_{k=1}^{n}{R_{k}^{2}}$$ and
$${B_{n}}=\sum\limits_{k=1}^{n}{{{R}_{k}}}{{R}_{k-2}}$$

$${{R}_{n}}\left( {{R}_{n-1}}+{{R}_{n-2}}+{{R}_{n}} \right)={H_{n}}+{B_{n}}+{{R}_{0}}({{R}_{0}}+{{R}_{-1}}+{{R}_{-2}})$$
\begin{equation}\label{e3.1}
{{R}_{n}}{{R}_{n+1}}={H_{n}}+{B_{n}}+{{R}_{0}}{{R}_{1}}.
\end{equation}
Again consider the relation from (\ref{e2})
$${{R}_{n}}-{{R}_{n-2}}={{R}_{n-1}}+{{R}_{n-3}}$$
$${{\left( {{R}_{n}}-{{R}_{n-2}} \right)}^{2}}={{\left( {{R}_{n-1}}+{{R}_{n-3}} \right)}^{2}}.$$
Taking sum on both sides.
$$\sum\limits_{k=1}^{n}{R_{k}^{2}}+\sum\limits_{k=1}^{n}{R_{k-2}^{2}}-2\sum\limits_{k=1}^{n}{{{R}_{k}}{{R}_{k-2}}}=\sum\limits_{k=1}^{n}{R_{k-1}^{2}}+\sum\limits_{k=1}^{n}{R_{k-3}^{2}}+2\sum\limits_{k=1}^{n}{{{R}_{k-1}}{{R}_{k-3}}}$$

$$4{B_{n}}+{{\left( {{R}_{-2}}+{{R}_{0}} \right)}^{2}}-{{\left( {{R}_{n}}+{{R}_{n-2}} \right)}^{2}}=0$$

$${B_{n}}=\frac{{{\left( {{R}_{n}}+{{R}_{n-2}} \right)}^{2}}-{{\left( {{R}_{-2}}+{{R}_{0}} \right)}^{2}}}{4}.$$
By equation (\ref{e2})
$${B_{n}}=\frac{{{\left( {{R}_{n+1}}-{{R}_{n-1}} \right)}^{2}}-{{\left( {{R}_{-2}}+{{R}_{0}} \right)}^{2}}}{4}.$$
Using the equation (\ref{e3.1}), we have
$$H_{n}=\frac{4{{R}_{n}}{{R}_{n+1}}-4{{R}_{0}}{{R}_{1}}-{{\left( {{R}_{n+1}}-{{R}_{n-1}} \right)}^{2}}+{{\left( {{R}_{-2}}+{{R}_{0}} \right)}^{2}}}{4}$$
\end{proof}

\begin{lem}\label{l.22}
For any $n\geq1$ \tiny
$$\sum\limits_{k=1}^{n}{{{R}_{k}}{{R}_{k+1}}}=C_{n}=\frac{{{B}_{n}}-{{H}_{n}}+R_{n}^{2}+R_{n-1}^{2}+R_{n-2}^{2}+{{R}_{n-2}}\left( {{R}_{n-3}}+2{{R}_{n-1}}-{{R}_{n}} \right)+2{{R}_{n}}\left( {{R}_{n-1}}+2{{R}_{n+1}} \right)+a\left( {{R}_{-1}}+c-2b \right)}{2},$$
\normalsize where $H_{n}$ and $B_{n}$ are defined lemma (\ref{l.2}).

\end{lem}
\begin{proof}
$${{{R}}_{k-1}}{{{R}}_{k-3}}=\left( {{{R}}_{k-2}}+{{{R}}_{k-3}}+{{{R}}_{k-4}} \right){{{R}}_{k-3}}$$
$${{{R}}_{k-1}}{{{R}}_{k-3}}={{{R}}_{k-2}}{{{R}}_{k-3}}+{R}_{k-3}^{2}+{{{R}}_{k-4}}{{{R}}_{k-3}}$$
$$\sum\limits_{k=3}^{n}{{{{R}}_{k-1}}{{{R}}_{k-3}}-\sum\limits_{k=3}^{n}{{R}_{k-3}^{2}}=\sum\limits_{k=3}^{n}{{{{R}}_{k-2}}{{{R}}_{k-3}}+\sum\limits_{k=3}^{n}{{{{R}}_{k-4}}{{{R}}_{k-3}}}}}$$
$$\left( {{B}_{n}}-{{R}_{n}}{{R}_{n-2}}+{{R}_{2}}{{R}_{0}} \right)-\left( {{H}_{n}}-R_{n}^{2}-R_{n-1}^{2}-R_{n-2}^{2} \right)=2\sum\limits_{k=3}^{n}{{{R}_{k-2}}{{R}_{k-3}}-{{R}_{n-2}}{{R}_{n-3}}-{{R}_{-1}}{{R}_{0}}}$$
$$\sum\limits_{k=3}^{n}{{{R}_{k-2}}{{R}_{k-3}}}=\frac{{{B}_{n}}-{{R}_{n}}{{R}_{n-2}}-{{H}_{n}}+{{R}_{-1}}{{R}_{0}}+{{R}_{2}}{{R}_{0}}+R_{n}^{2}+R_{n-1}^{2}+R_{n-2}^{2}+{{R}_{n-2}}{{R}_{n-3}}}{2}$$
$$\sum\limits_{k=3}^{n}{{{R}_{k-2}}{{R}_{k-3}}}=\frac{{{B}_{n}}-{{R}_{n}}{{R}_{n-2}}-{{H}_{n}}+{{R}_{-1}}a+ac+R_{n}^{2}+R_{n-1}^{2}+R_{n-2}^{2}+{{R}_{n-2}}{{R}_{n-3}}}{2}$$
Since
$$\sum\limits_{k=3}^{n}{{{R}_{k-2}}{{R}_{k-3}}=}\sum\limits_{k=1}^{n}{{{R}_{k}}{{R}_{k+1}}}-{{R}_{n-1}}{{R}_{n-2}}-{{R}_{n}}{{R}_{n-1}}-{{R}_{n+1}}{{R}_{n}}+{{R}_{0}}{{R}_{1}}$$

$$=\sum\limits_{k=1}^{n}{{{R}_{k}}{{R}_{k+1}}}-{{R}_{n-1}}{{R}_{n-2}}-{{R}_{n}}{{R}_{n-1}}-{{R}_{n+1}}{{R}_{n}}+ab$$
\tiny
$$=\frac{{{B}_{n}}-{{R}_{n}}{{R}_{n-2}}-{{A}_{n}}+{{R}_{-1}}a+ac+R_{n}^{2}+R_{n-1}^{2}+R_{n-2}^{2}+{{R}_{n-2}}{{R}_{n-3}}+2{{R}_{n-1}}{{R}_{n-2}}+2{{R}_{n}}{{R}_{n-1}}+2{{R}_{n+1}}{{R}_{n}}-2ab}{2}$$
$$=\frac{{{B}_{n}}-{{R}_{n}}{{R}_{n-2}}-{{A}_{n}}+{{R}_{-1}}a+ac+R_{n}^{2}+R_{n-1}^{2}+R_{n-2}^{2}+{{R}_{n-2}}{{R}_{n-3}}+2{{R}_{n-1}}{{R}_{n-2}}+2{{R}_{n}}{{R}_{n-1}}+2{{R}_{n+1}}{{R}_{n}}-2ab}{2}$$
$$C_{n}=\frac{{{B}_{n}}-{{A}_{n}}+R_{n}^{2}+R_{n-1}^{2}+R_{n-2}^{2}+{{R}_{n-2}}\left( {{R}_{n-3}}+2{{R}_{n-1}}-{{R}_{n}} \right)+2{{R}_{n}}\left( {{R}_{n-1}}+2{{R}_{n+1}} \right)+a\left( {{R}_{-1}}+c-2b \right)}{2}$$

\end{proof}
\begin{lem}\label{ll.2}
\tiny
For any $n\geq1$
$$\sum\limits_{k=1}^{n}{\sum\limits_{i=1}^{k}{R_{i}^{2}}}={{M}_{n}}=\frac{4{{C}_{n}}-2{{A}_{n}}+2{{B}_{n}}-4nab-{{a}^{2}}+{{b}^{2}}-2b{{R}_{-1}}+R_{n}^{2}-{{R}_{n+1}}\left( 1-2{{R}_{n-1}} \right)+n{{\left( {{R}_{-2}}+a \right)}^{2}}}{4},$$
\normalsize where $H_{n}$, $B_{n}$ and $C_{n}$ are defined in lemmas
(\ref{l.2}) and (\ref{l.22}) respectively.
\end{lem}

\begin{proof}
From lemma (\ref{l.2}), we have
$$\sum\limits_{k=1}^{n}{R_{k}^{2}}=\frac{4{{R}_{n}}{{R}_{n+1}}-4{{R}_{0}}{{R}_{1}}-{{\left( {{R}_{n+1}}-{{R}_{n-1}} \right)}^{2}}+{{\left( {{R}_{-2}}+{{R}_{0}} \right)}^{2}}}{4}$$
$$\sum\limits_{k=1}^{n}{R_{k}^{2}}=\frac{4{{R}_{n}}{{R}_{n+1}}-4{{R}_{0}}{{R}_{1}}-R_{n+1}^{2}-R_{n-1}^{2}+2{{R}_{n+1}}{{R}_{n-1}}+{{\left( {{R}_{-2}}+{{R}_{0}} \right)}^{2}}}{4}$$
$$\sum\limits_{k=1}^{n}{\sum\limits_{i=1}^{k}{R_{i}^{2}}}=\frac{4\sum\limits_{k=1}^{n}{{{R}_{n}}{{R}_{n+1}}}-4\sum\limits_{k=1}^{n}{{{a}}{b}}-\sum\limits_{k=1}^{n}{R_{n+1}^{2}}-\sum\limits_{k=1}^{n}{R_{n-1}^{2}}+2\sum\limits_{k=1}^{n}{{{R}_{n+1}}{{R}_{n-1}}\,}+\sum\limits_{k=1}^{n}{{{\left( {{R}_{-2}}+{a} \right)}^{2}}}}{4}$$

$$=\frac{4C_{n}-4n{a}{b}-\left( H_{n}+{{R}_{n+1}}-{b}^{2} \right)-\left( A_{n}+{a}^{2}-R_{n}^{2} \right)+2\left( B_{n}-{b}{{R}_{-1}}+{{R}_{n+1}}{{R}_{n-1}} \right)+n{{\left( {{R}_{-2}}+{a} \right)}^{2}}}{4}$$
$$=\frac{4C_{n}-2H_{n}+2B_{n}-4n{a}{b}-{a}^{2}+R_{n}^{2}-{{R}_{n+1}}+{b}^{2}-2{b}{{R}_{-1}}+2{{R}_{n+1}}{{R}_{n-1}}+n{{\left( {{R}_{-2}}+{a} \right)}^{2}}}{4}$$
\end{proof}
\begin{thm}
Let ${A}={{{A}}_{r}}({{R}_{0}},{{R}_{1,}}...,{{R}_{n-1}})$
be $r-$circulant matrix. \\

If $\left| r \right|\ge 1$, then  $\sqrt{\sum\limits_{k=0}^{n-1}{R_{k}^{2}}}\le {{\left\|
{A} \right\|}_{2}}\le \sqrt{\left( {{R}_{0}}^{2}+|r{{|}^{2}}\sum\limits_{k=1}^{n-1}{R_{k}^{2}} \right)\left( 1+\sum\limits_{k=1}^{n-1}{R_{k}^{2}} \right)}$

If $\left| r \right|< 1$, then $\left| r \right|\sqrt{\sum\limits_{k=0}^{n-1}{R_{k}^{2}}}\le
{{\left\| {A} \right\|}_{2}}\le \sqrt{n\left(\frac{4{{R}_{n-1}}{{R}_{n}}-4{{R}_{0}}{{R}_{1}}-{{\left( {{R}_{n}}-{{R}_{n-2}} \right)}^{2}}+{{\left( {{R}_{-2}}+{{R}_{0}} \right)}^{2}}+4R_{0}^{2}}{4}\right)}.$
\end{thm}

\begin{proof}
The $r-$circulant matrix ${A}$ on the sequence (\ref{e2}) is given
as,
$${A}=\left[ \begin{matrix}
   {{R}_{0}} & {{R}_{1}} & {{R}_{2}} & \cdots   & {{R}_{n-1}}  \\
   r{{R}_{n-1}} & {{R}_{0}} & {{R}_{1}} & \cdots    & {{R}_{n-2}}  \\
   r{{R}_{n-2}} & r{{R}_{n-1}} & {{R}_{0}} & \cdots   & {{R}_{n-3}}  \\
   \vdots  & \vdots  & \vdots  & \ddots  & \vdots   \\
   r{{R}_{1}} & r{{R}_{2}} & r{{R}_{3}} & \cdots    & {{R}_{0}}  \\
\end{matrix} \right]$$
and from the definition of Euclidean norm, we have
\begin{equation}\label{e3.2}
\left\| {A} \right\|_{E}^{2}=\sum\limits_{k=0}^{n-1}{\left(
n-k \right)}\,R_{k}^{2}+\sum\limits_{k=1}^{n-1}{k{{\left| r
\right|}^{2}}}R_{k}^{2}.
\end{equation}
 Here we have two cases depending on $r$.\\

Case 1.If  $\left| r \right|\ge 1$, then from equation (\ref{e3.2}), we
have
$$\left\| {A} \right\|_{E}^{2}\ge \sum\limits_{k=0}^{n-1}{\left( n-k
\right)}R_{k}^{2}+\sum\limits_{k=1}^{n-1}{k}R_{k}^{2}=n\sum\limits_{k=0}^{n-1}{R_{k}^{2}},$$

$$\left\| {A} \right\|_{E}^{2}\ge n\sum\limits_{k=0}^{n-1}{R_{k}^{2}}$$
By inequality (\ref{e6}), we obtain
\begin{equation}\label{r.1}
    {{\left\| {A} \right\|}_{2}}\ge \sqrt{\sum\limits_{k=0}^{n-1}{R_{k}^{2}}}.
\end{equation}
On the other hand, let us define two new  matrices ${ C}$ and ${D }$  as :\\\

${C}=\left[ \begin{matrix}
   {{R}_{0}} & 1 & 1 & \cdots  & 1  \\
   r{{R}_{n-1}} & r{{R}_{0}} & 1 & \cdots  & 1  \\
   r{{R}_{n-2}} & r{{R}_{n-1}} & r{{R}_{0}} & \cdots  & 1  \\
   \vdots  & \vdots  & \vdots  & \ddots  & \vdots   \\
   r{{R}_{1}} & r{{R}_{2}} & r{{R}_{3}} & \cdots  & {{R}_{0}}  \\
\end{matrix} \right]$\, and \,${D}=\left[ \begin{matrix}
   {1} & {{R}_{1}} & {{R}_{2}} & \cdots  & {{R}_{n-1}}  \\
   1 & {{R}_{0}} & {{R}_{1}} & \cdots  & {{R}_{n-2}}  \\
   1 & 1 & {{R}_{0}} & \cdots  & {{R}_{n-3}}  \\
   \vdots  & \vdots  & \vdots  & \ddots  & \vdots   \\
   1 & 1 & 1 & \cdots  & {1}  \\
\end{matrix} \right]$\\\

Then it is easy to see that  ${A}={C}\circ {D}$ ,so from definition
(\ref{d.3})
$${{r}_{1}}(C)=\underset{i\le i\le n}{\mathop{\max }}\,\sqrt{\sum\limits_{j=1}^{n}{{{\left| \text{ }{{c}_{ij}} \right|}^{2}}}}=\sqrt{{{R}_{0}}^{2}+|r{{|}^{2}}\sum\limits_{k=1}^{n-1}{R_{k}^{2}}}$$
and
$${{{c}}_{1}}({D})=\underset{1\le j\le n}{\mathop{\max
}}\,\sqrt{\sum\limits_{i=1}^{n}{{{\left| {{{d}}_{ij}}
\right|}^{2}}}}=\sqrt{1+\sum\limits_{k=1}^{n-1}{R_{k}^{2}}}$$ Now
using theorem (\ref{t.1}), we obtain
 $${{\left\| {A} \right\|}_{2}}\le
{{r}_{1}}({C}){{c}_{1}}({D})=\sqrt{\left( {{R}_{0}}^{2}+|r{{|}^{2}}\sum\limits_{k=1}^{n-1}{R_{k}^{2}} \right)\left( 1+\sum\limits_{k=1}^{n-1}{R_{k}^{2}} \right)}$$
\begin{equation}\label{r.2}
{{\left\| {A} \right\|}_{2}}\le \sqrt{\left( {{R}_{0}}^{2}+|r{{|}^{2}}\sum\limits_{k=1}^{n-1}{R_{k}^{2}} \right)\left( 1+\sum\limits_{k=1}^{n-1}{R_{k}^{2}} \right)}.
\end{equation}
Combine equations (\ref{r.1}) and (\ref{r.2}), we get following inequality
$$\sqrt{\sum\limits_{k=0}^{n-1}{R_{k}^{2}}}\le {{\left\|
{A} \right\|}_{2}}\le \sqrt{\left( {{R}_{0}}^{2}+|r{{|}^{2}}\sum\limits_{k=1}^{n-1}{R_{k}^{2}} \right)\left( 1+\sum\limits_{k=1}^{n-1}{R_{k}^{2}} \right)}.$$ Case 2. If
$\left| r \right|\le 1$, then we have
$$\left\| {A} \right\|_{E}^{2}\ge \sum\limits_{k=0}^{n-1}{\left( n-k
\right)}\,{{\left| r
\right|}^{2}}R_{k}^{2}+\sum\limits_{k=0}^{n-1}{k{{\left| r
\right|}^{2}}}R_{k}^{2}=n\sum\limits_{k=0}^{n-1}{{{\left| r
\right|}^{2}}R_{k}^{2}}$$
$$\frac{1}{\sqrt{n}}{{\left\| {A} \right\|}_{E}}\ge |r|\sqrt{\sum\limits_{k=0}^{n-1}{R_{k}^{2}}}.$$
By inequality  (\ref{e6}), we get \\
\begin{equation}\label{r.3}
    \left\| {A} \right\|{}_{2}\ge \left| r \right|\sqrt{\sum\limits_{k=0}^{n-1}{R_{k}^{2}}}.
\end{equation}
On the other hand, let the matrices ${{C}}'$ and ${{D}}'$ be defined as: \\\

${{C}}'=\left[ \begin{matrix}
   {1} & 1 & 1 & \cdots  & 1  \\
   r & {1} & 1 & \cdots  & 1  \\
   r & r & {1} & \cdots  & 1  \\
   \vdots  & \vdots  & \vdots  & \ddots  & \vdots   \\
   r & r & r & \cdots  & {1}  \\
\end{matrix} \right]$  and   ${{D}}'=\left[ \begin{matrix}
   {{R}_{0}} & {{R}_{1}} & {{R}_{2}} & \cdots  & {{R}_{n-1}}  \\
   {{R}_{n-1}} & {{R}_{0}} & {{R}_{1}} & \cdots   & {{R}_{n-2}}  \\
   {{R}_{n-2}} & {{R}_{n-1}} & {{R}_{0}} & \cdots & {{R}_{n-3}}  \\
   \vdots  & \vdots  & \vdots   & \ddots  & \vdots   \\
   {{R}_{1}} & {{R}_{2}} & {{R}_{3}} & \cdots  & {{R}_{0}}  \\
\end{matrix} \right]$\\\\

such that ${A}={{C}}'\circ {{D}}'$, then by definition (1.2), we obtain
$${{r}_{1}}({{C}}')=\underset{1\le i\le n}{\mathop{\max
}}\,\sqrt{\sum\limits_{j=1}^{n}{{{\left| {{{{{c}}'}}_{ij}}
\right|}^{2}}}}=\sqrt{n}$$
and
\small
$${{c}_{1}}({{D}}')=\underset{1\le j\le n}{\mathop{\max
}}\,\sqrt{\sum\limits_{i=1}^{n}{{{\left| {{{{{d}}'}}_{ij}}
\right|}^{2}}}}=\sqrt{\sum\limits_{k=0}^{n-1}{R_{k}^{2}}}=\sqrt{\left(\frac{4{{R}_{n-1}}{{R}_{n}}-4ab-{{\left( {{R}_{n}}-{{R}_{n-2}} \right)}^{2}}+{{\left( {{R}_{-2}}+{a} \right)}^{2}}+4{a}^{2}}{4}\right)}.$$
\normalsize
Again by applying (1.3), we get\\
$${{\left\| {A} \right\|}_{2}}\le
{{r}_{1}}({{C}}'){{c}_{1}}({{D}}')=\sqrt{n\left(\frac{4{{R}_{n-1}}{{R}_{n}}-4ab-{{\left( {{R}_{n}}-{{R}_{n-2}} \right)}^{2}}+{{\left( {{R}_{-2}}+{a} \right)}^{2}}+4{a}^{2}}{4}\right)},$$
\begin{equation}\label{r.4}
{{\left\| {A} \right\|}_{2}}\le \sqrt{n\left(\frac{4{{R}_{n-1}}{{R}_{n}}-4{{R}_{0}}{{R}_{1}}-{{\left( {{R}_{n}}-{{R}_{n-2}} \right)}^{2}}+{{\left( {{R}_{-2}}+{{R}_{0}} \right)}^{2}}+4R_{0}^{2}}{4}\right)},
\end{equation}
and combing inequality (\ref{r.3}) and (\ref{r.4}), we obtain the required result. \\
$$\left| r \right|\sqrt{\sum\limits_{k=0}^{n-1}{R_{k}^{2}}}\le
{{\left\| {A} \right\|}_{2}}\le \sqrt{n\left(\frac{4{{R}_{n-1}}{{R}_{n}}-4{{R}_{0}}{{R}_{1}}-{{\left( {{R}_{n}}-{{R}_{n-2}} \right)}^{2}}+{{\left( {{R}_{-2}}+{{R}_{0}} \right)}^{2}}+4R_{0}^{2}}{4}\right)}.$$
\end{proof}

\begin{thm}
Let ${A}$ be the circulant matrix on generalized tribonacci
sequence. Then,  $${{\left\| {A} \right\|}_{E}}=\sqrt{n\left(\frac{4{{R}_{n-1}}{{R}_{n}}-4ab-{{\left( {{R}_{n}}-{{R}_{n-2}} \right)}^{2}}+{{\left( {{R}_{-2}}+{a} \right)}^{2}}+4{a}^{2}}{4}\right)}$$\\
and\\
 $\sqrt{\left(\frac{4{{R}_{n-1}}{{R}_{n}}-4ab-{{\left( {{R}_{n}}-{{R}_{n-2}} \right)}^{2}}+{{\left( {{R}_{-2}}+{a} \right)}^{2}}+4{a}^{2}}{4}\right)
\le }{{\left\|{ A} \right\|}_{2}}\le \sqrt{\left( \sum\limits_{i=0}^{n-1}{R_{i}^{2}} \right)\left( 1+\sum\limits_{i=1}^{n-1}{R_{i}^{2}} \right)}.$\\
\end{thm}
\begin{proof}
Since by definition of circulant Matrix, the matrix ${U}$ is of the form  \\
$${A}=\left[ \begin{matrix}
   {{R}_{0}} & {{R}_{1}} & {{R}_{2}} & \cdots  & {{R}_{n-1}}  \\
   {{R}_{n-1}} & {{R}_{0}} & {{R}_{1}} & \cdots   & {{R}_{n-2}}  \\
   {{R}_{n-2}} & {{R}_{n-1}} & {{R}_{0}} & \cdots    & {{R}_{n-3}}  \\
     \vdots  & \vdots  & \vdots  & \ddots  & \vdots  &   \\
   {{R}_{1}} & {{R}_{2}} & {{R}_{3}} & \cdots    & {{R}_{0}}  \\
\end{matrix} \right].$$
By definition of Euclidean norm, we have \\
\begin{equation}
    {{\left\|{ A}
\right\|}_{E}}=\sqrt{n\sum\limits_{i=0}^{n-1}{R_{i}^{2}}}=\sqrt
n\sqrt{\left(\frac{4{{R}_{n-1}}{{R}_{n}}-4ab-{{\left( {{R}_{n}}-{{R}_{n-2}} \right)}^{2}}+{{\left( {{R}_{-2}}+{a} \right)}^{2}}+4{a}^{2}}{4}\right)}.
\end{equation}
By inequality (\ref{e6}), we get
\begin{equation}\label{r.5}
 \sqrt{\left(\frac{4{{R}_{n-1}}{{R}_{n}}-4ab-{{\left( {{R}_{n}}-{{R}_{n-2}} \right)}^{2}}+{{\left( {{R}_{-2}}+{{a}} \right)}^{2}}+4R_{0}^{2}}{4}\right)}\le {{\left\| {A} \right\|}_{2}}.
\end{equation}
Let matrices ${B}$ and ${C}$ be defined as:\\

$ {B} = \left\{
        \begin{array}{ll}
            {{{b}}_{ij}}={{R}_{\left( \,\bmod \,\left( j-i,n \right) \right)}} & \quad i \geq j \\
            {{{b}}_{ij}}=1 & \quad i < j
        \end{array}
    \right.
$ and $ {C} = \left\{
        \begin{array}{ll}
            {{{c}}_{ij}}={{R}_{\left( \,\bmod \,\left( j-i,n \right) \right)}} & \quad i < j \\
            {{{c}}_{ij}}=1 & \quad i \geq j
        \end{array}
    \right.
$\\

Then the row norm and column norm of $ {B}$ and $ {C}$
are given as:\\
$${{r}_{1}}({B})=\underset{i}{\mathop{\max
}}\,\sqrt{\sum\limits_{j=1}^{n}{\left| {{{b}}_{ij}}
\right|{}^{2}}}=\sqrt{\sum\limits_{i=0}^{n-1}{R_{i}^{2}}}$$
\normalsize
and
$${{c}_{1}}({C})=\underset{j}{\mathop{\max
}}\,\sqrt{\sum\limits_{i=1}^{n}{\left| {{{c}}_{ij}}
\right|{}^{2}}}=\sqrt{1+\sum\limits_{i=1}^{n-1}{R_{i}^{2}}}.$$
Using theorem  (\ref{t.1}), we have \\
\begin{equation}\label{r.6}
  {{\left\| A \right\|}_{2}}\le \sqrt{\left( \sum\limits_{i=0}^{n-1}{R_{i}^{2}} \right)\left( 1+\sum\limits_{i=1}^{n-1}{R_{i}^{2}} \right)}.
\end{equation}
Combine (\ref{r.5}) and (\ref{r.6}), we get \\
\small
$$ \sqrt{\left(\frac{4{{R}_{n-1}}{{R}_{n}}-4ab-{{\left( {{R}_{n}}-{{R}_{n-2}} \right)}^{2}}+{{\left( {{R}_{-2}}+{a} \right)}^{2}}+4{a}^{2}}{4}\right)
\le }{{\left\|{ A} \right\|}_{2}}\le \sqrt{\left( \sum\limits_{i=0}^{n-1}{R_{i}^{2}} \right)\left( 1+\sum\limits_{i=1}^{n-1}{R_{i}^{2}} \right)}.$$
\end{proof}
\normalsize

\begin{thm}
Let $A$ be an $n\times n$ semi-circulant matrix ${A}=({a}_{ij})$
with the generalized tribonacci sequence then,
\small
$$\left\| {A} \right\|_{E}^{2}=\frac{4{{C}_{n}}-2{{A}_{n}}+2{{B}_{n}}-4nab-{{a}^{2}}+{{b}^{2}}-2b{{R}_{-1}}+R_{n}^{2}-{{R}_{n+1}}\left( 1-2{{R}_{n-1}} \right)+n{{\left( {{R}_{-2}}+a \right)}^{2}}}{4},$$
\end{thm}
\normalsize
\begin{proof}
For the semi-circulant matrix ${A}=({a}_{ij})$ with
the generalized tribonacci sequence
sequence, we have \\
$${{{a}}_{ij}}=\left\{ \begin{array}{*{35}{l}}
   {{R}_{j-i+1}} & \quad i\le j  \\
   0 & \quad otherwise . \\
\end{array} \right.$$
From the definition of Euclidean norm, we have. \\
$${{\left\| {A}
\right\|}_{E}^2}=\sum\limits_{j=1}^{n}{\sum\limits_{i=1}^{j}{{{\left(
{{R}_{j-i+1}} \right)}^{2}}}}=
 \sum\limits_{j=1}^{n}{\left( \sum\limits_{k=1}^{j}{R_{k}^{2}}
 \right)}.$$\\
Using lemma (1.9), we get the required result.

 $$\left\| {A} \right\|_{E}^{2}=\frac{4{{C}_{n}}-2{{A}_{n}}+2{{B}_{n}}-4nab-{{a}^{2}}+{{b}^{2}}-2b{{R}_{-1}}+R_{n}^{2}-{{R}_{n+1}}\left( 1-2{{R}_{n-1}} \right)+n{{\left( {{R}_{-2}}+a \right)}^{2}}}{4}.$$

\end{proof}
\normalsize

\begin{thm}\label{h.1}
 If ${A}= ({{{a}}_{ij}} )$ is an $n\times n$ Hankel matrix with
${{{a}}_{ij}}={{R}_{i+j-1}}$, then
$${{\left\| \text{ }A \right\|}_{E}}={{\left( {{M}_{2n-1}}-2{{M}_{n-1}} \right)}^{\frac{1}{2}}}.$$
where ${M}_{n}$ is defined in lemma (\ref{ll.2}).
\end{thm}
\begin{proof}
From the definition of Hankel matrix, the matrix ${A}$ is of the form  \\
$$
{A}=\left[ \begin{matrix}
   {{R}_{1}} & {{R}_{2}} & {{R}_{3}} & \cdots  & {{R}_{n-1}} & {{R}_{n}}  \\
   {{R}_{2}} & {{R}_{3}} & {{R}_{4}} & \cdots  & {{R}_{n}} & {{R}_{n+1}}  \\
   {{R}_{3}} & {{R}_{4}} & {{R}_{5}} & \cdots  & {{R}_{n+1}} & {{R}_{n+2}}  \\
   \vdots  & \vdots  & \vdots  & \ddots  & \vdots  & \vdots   \\
   {{R}_{n-1}} & {{R}_{n}} & {{R}_{n+1}} & \cdots  & {{R}_{2n-3}} & {{R}_{2n-2}}  \\
   {{R}_{n}} & {{R}_{n+1}} & {{R}_{n+2}} & \cdots  & {{R}_{2n-2}} & {{R}_{2n-1}}  \\
\end{matrix} \right].$$

So, we have
  $${{\left\| {A} \right\|}_{E}}={{\left(
\sum\limits_{i=1}^{m}{\sum\limits_{j=1}^{n}{{{\left|
{{{a}}_{ij}} \right|}^{2}}}} \right)}^{{}^{1}/{}_{2}}}.$$

$${{\left\|{ A} \right\|}_{E}}={{\left(
\sum\limits_{k=1}^{n}{R_{k}^{2}+\sum\limits_{k=2}^{n+1}{R_{k}^{2}}+...+\sum\limits_{k=n}^{2n-1}{R_{k}^{2}}}
\right)}^{{}^{1}/{}_{2}}}.$$

$${{\left\| {A} \right\|}_{E}}={{\left( \left(
\sum\limits_{k=1}^{n}{R_{k}^{2}+\sum\limits_{k=1}^{n+1}{R_{k}^{2}}+...+\sum\limits_{k=1}^{2n-1}{R_{k}^{2}}}
\right)-\left(
\sum\limits_{k=1}^{n-1}{\sum\limits_{i=1}^{k}{R_{i}^{2}}} \right)
\right)}^{{}^{1}/{}_{2}}}.$$

$${{\left\| \text{ }A \right\|}_{E}}={{\left( \left( {{H}_{n}} \right)+({{H}_{n+1}})+...+\left( {{H}_{2n-1}} \right)-\sum\limits_{k=1}^{n-1}{\left( {{H}_{n}} \right)} \right)}^{\frac{1}{2}}}$$
$${{\left\| \text{ }A \right\|}_{E}}={{\left( \sum\limits_{k=n}^{2n-1}{\left( {{H}_{k}} \right)}-\sum\limits_{k=1}^{n-1}{\left( {{H}_{n}} \right)} \right)}^{\frac{1}{2}}}.$$
$${{\left\| \text{ }A \right\|}_{E}}={{\left( \sum\limits_{k=1}^{2n-1}{\left( {{H}_{k}} \right)}-2\sum\limits_{k=1}^{n-1}{\left( {{H}_{k}} \right)} \right)}^{\frac{1}{2}}}.$$
$${{\left\| \text{ }A \right\|}_{E}}={{\left( \sum\limits_{k=1}^{2n-1}{\sum\limits_{i=1}^{k}{R_{i}^{2}}}-2\sum\limits_{k=1}^{n-1}{\sum\limits_{i=1}^{k}{R_{i}^{2}}} \right)}^{\frac{1}{2}}}.$$
$${{\left\| \text{ }A \right\|}_{E}}={{\left( {{M}_{2n-1}}-2{{M}_{n-1}} \right)}^{\frac{1}{2}}}.$$
where $M_{n}$ is defined in lemma (\ref{ll.2})
\end{proof}

\begin{thm}
If ${A}= ({{{a}}_{ij}} )$ is an $n\times n$ Hankel
matrix with ${{{a}}_{ij}}={{R}_{i+j-1}}$ then, we have
$$\frac{1}{\sqrt{n}}{{\left\| {A} \right\|}_{E}}\le{\left\| {A}
\right\|}_{2}\leq \sqrt{\left( \sum\limits_{i=1}^{n}{R_{i}^{2}} \right)\left( 1+\sum\limits_{i=2}^{n}{R_{i}^{2}} \right)}.$$
\end{thm}
\begin{proof}
From theorem (\ref{h.1}) and inequality (\ref{e6}), we have
$$\frac{1}{\sqrt{n}}{{\left\| {A} \right\|}_{E}}\le {{\left\| {A}
\right\|}_{2}}.$$

Let us define two new matrices\\\

 ${{M}}= \left\{
        \begin{array}{ll}
          {m}_{ij}={{R}_{i+j-1}} & i\leq j\\
          {n}_{ij}=1 & i> j\,
          \,\,
        \end{array}
    \right.
$ and\, ${{N}}= \left\{
        \begin{array}{ll}
          {n}_{ij}={{R}_{i+j-1}} & i> j\,\\
          {n}_{ij}=1 & i\leq j.
          \,\,
        \end{array}
    \right.
$\\

It can be  easily seen that ${A}={{M}}\circ
{{N}}$. Thus we get\\

$${{r}_{1}}\left( {{M}} \right)=\underset{i}{\mathop{\max
}}\,\sqrt{\sum\limits_{j}{{{\left| {{{m}}_{ij}}
\right|}^{2}}}}=\sqrt{\sum\limits_{i=1}^{n}{R_{i}^{2}}}$$\\
and\\
$${{c}_{1}}\left( {{N}} \right)=\underset{j}{\mathop{\max
}}\,\sqrt{\sum\limits_{i}{{{\left| {{{n}}_{ij}}
\right|}^{2}}}}=\sqrt{1+\sum\limits_{i=2}^{n}{R_{i}^{2}}}.$$

Using the theorem (\ref{t.1}), we have
$${{\left\| {A} \right\|}_{2}}\le \sqrt{\left( \sum\limits_{i=1}^{n}{R_{i}^{2}} \right)\left( 1+\sum\limits_{i=2}^{n}{R_{i}^{2}} \right)} .$$
\end{proof}
\begin{thm}
If ${A}= ({{{a}}_{ij}} )$ is an $n\times n$ Hankel
matrix with
${{{a}}_{ij}}={{U}_{i+j-1}}$.\\
Then, ${{\left\| {A} \right\|}_{1}}={{\left\| {A} \right\|}_{\infty
}}={{\left\| {A} \right\|}_{1}}={{\left\| {A} \right\|}_{\infty
}}=\left( \frac{\,{{R}_{2n+2}}-{{R}_{2n}}-{{R}_{n+2}}+{{R}_{n}}}{2}
\right).$
\end{thm}

\begin{proof}
From the definition of the matrix ${A}$ , we can write\\
$${{\left\| {A} \right\|}_{1}}=\underset{i\le j\le n}{\mathop{\max
}}\,\sum\limits_{i=1}^{n}{\left| {{{a}}_{ij}}
\right|}=\underset{1\le j\le n}{\mathop{\max }}\,\left\{ \left|
{{{a}}_{1j}} \right|+\left| {{{a}}_{2j}}
\right|+\left| {{{a}}_{3j}} \right|...\left|
{{{a}}_{nj}} \right| \right\}$$
$${{\left\|{A}
\right\|}_{1}}={{R}_{n}}+{{R}_{n+1}}+{{R}_{n+2}}+\cdots+{{R}_{2n-1}}$$
$${{\left\| {A}
\right\|}_{1}}=\sum\limits_{i=1}^{2n-1}{{{R}_{i}}}-\sum\limits_{i=1}^{n-1}{{{R}_{i}}},$$
by Corolarry (3.2)   \\
$${{\left\| {A} \right\|}_{1}}={{\left\| {A} \right\|}_{\infty
}}=\left( \frac{\,{{R}_{2n+2}}-{{R}_{2n}}-{{R}_{n+2}}+{{R}_{n}}}{2}
\right).$$\\
Similarly the row norm of the matrix ${A}$ can be computed as: \\\
$${{\left\|{ A} \right\|}_{\infty }}=\underset{1\le i\le
n}{\mathop{\max }}\,\sum\limits_{j=1}^{n}{\left| {{{a}}_{ij}}
\right|}=\left(
\frac{\,{{R}_{2n+2}}-{{R}_{2n}}-{{R}_{n+2}}+{{R}_{n}}}{2}
\right).$$\\
\end{proof}

\section{Generalized Pell Padovan sequence}
This section consist of some important properties of Generalized pell Padovan sequence and some main results.
\begin{lem}\cite{Arzu}\label{l3.1}
Sum of first n terms of the sequence (\ref{e3})\\
$$\sum\limits_{k=1}^{n}{{{Z}_{k}}={{Z}_{n+5}}-{{Z}_{4}}}-{{Z}_{0}}$$
\end{lem}
\begin{lem}\cite{Arzu}\label{l3.2}
Sum of square first n terms of the sequence (\ref{e3})\\
$$\textbf{Z}_{n}=\sum\limits_{k=1}^{n}{Z_{k}^{2}}=Z_{n+2}^{2}-Z_{n-1}^{2}-Z_{n-3}^{2}+T-Z_{0}^{2}$$
where $T=2a\left( a-c \right)-{{\left( b-c \right)}^{2}}.$
\end{lem}

\begin{lem}\label{l4.1}
For all $n\geq1$\\
$$\sum\limits_{m=1}^{n}{\sum\limits_{k=1}^{m}{Z_{k}^{2}}}=Z_{n+1}^{2}+2Z_{n}^{2}+2Z_{n-1}^{2}+Z_{n-2}^{2}+Z_{n-3}^{2}-Z_{-2}^{2}-Z_{-1}^{2}-{{b}^{2}}-{{c}^{2}}-2{{a}^{2}}+\left( n-1 \right)\left( T-{{a}^{2}} \right).$$
\end{lem}
\begin{proof}
From the lemma (\ref{l3.2})
$$\textbf{Z}_{n}=\sum\limits_{k=1}^{n}{Z_{k}^{2}}=Z_{n+2}^{2}-Z_{n-1}^{2}-Z_{n-3}^{2}+T-{{a}^{2}}$$

$$\sum\limits_{m=1}^{n}{\sum\limits_{k=1}^{m}{Z_{k}^{2}}}=\sum\limits_{m=1}^{n}{\left( Z_{m+2}^{2}-Z_{m-1}^{2}-Z_{m-3}^{2}+T-{{a}^{2}} \right)}$$
\begin{equation}
\sum\limits_{m=1}^{n}{\sum\limits_{k=1}^{m}{Z_{k}^{2}}}=\sum\limits_{m=1}^{n}{Z_{m+2}^{2}}-\sum\limits_{m=1}^{n}{Z_{m-1}^{2}}-\sum\limits_{m=1}^{n}{Z_{m-3}^{2}}+n\left( T-{{a}^{2}} \right)
\end{equation}
Where,
$$\sum\limits_{m=1}^{n}{Z_{m+2}^{2}}=\textbf{Z}_{n}+Z_{n+1}^{2}+Z_{n+2}^{2}-Z_{1}^{2}-Z_{2}^{2}$$

$$\sum\limits_{m=1}^{n}{Z_{m-1}^{2}}=\textbf{Z}_{n}-Z_{n}^{2}+Z_{0}^{2}$$

$$\sum\limits_{m=1}^{n}{Z_{m-3}^{2}}=\textbf{Z}_{n}-Z_{n-2}^{2}-Z_{n-1}^{2}-Z_{n}^{2}+Z_{-2}^{2}+Z_{-1}^{2}+Z_{0}^{2}$$

$$\sum\limits_{m=1}^{n}{\sum\limits_{k=1}^{m}{Z_{k}^{2}}}=Z_{n+1}^{2}+2Z_{n}^{2}+2Z_{n-1}^{2}+Z_{n-2}^{2}+Z_{n-3}^{2}-Z_{-2}^{2}-Z_{-1}^{2}-{{b}^{2}}-{{c}^{2}}-2{{a}^{2}}+\left( n-1 \right)\left( T-{{a}^{2}} \right).$$
\end{proof}
\begin{thm}
Let ${A}={{{A}}_{r}}({{Z}_{0}},{{Z}_{1,}}...,{{Z}_{n-1}})$ be
$r-$circulant matrix. \\

If $\left| r \right|\ge 1$, then  $\sqrt{\left(
Z_{n+1}^{2}-Z_{n-2}^{2}-Z_{n-4}^{2}+T \right)}\le {{\left\| {A} \right\|}_{2}}\le \sqrt{\left( Z_{0}^{2}+{{\left| r \right|}^{2}}\sum\limits_{k=1}^{n-1}{Z_{k}^{2}} \right)\left( 1+\sum\limits_{k=1}^{n-1}{Z_{k}^{2}} \right)}$

If $\left| r \right|< 1$, then $\left| r \right|\sqrt{\left(
Z_{n+1}^{2}-Z_{n-2}^{2}-Z_{n-4}^{2}+T \right)}\le
{{\left\| {A} \right\|}_{2}}\le \sqrt{n\left(
Z_{n+1}^{2}-Z_{n-2}^{2}-Z_{n-4}^{2}+T \right)}.$
\end{thm}

\begin{proof}
The $r-$circulant matrix ${A}$ on the sequence (\ref{e3}) is given
as,
$${A}=\left[ \begin{matrix}
   {{Z}_{0}} & {{Z}_{1}} & {{Z}_{2}} & \cdots   & {{Z}_{n-1}}  \\
   r{{Z}_{n-1}} & {{Z}_{0}} & {{Z}_{1}} & \cdots    & {{Z}_{n-2}}  \\
   r{{Z}_{n-2}} & r{{Z}_{n-1}} & {{Z}_{0}} & \cdots   & {{Z}_{n-3}}  \\
   \vdots  & \vdots  & \vdots  & \ddots  & \vdots   \\
   r{{Z}_{1}} & r{{Z}_{2}} & r{{Z}_{3}} & \cdots    & {{Z}_{0}}  \\
\end{matrix} \right]$$
and from the definition of Euclidean norm, we have
\begin{equation}\label{e4.1}
\left\| {A} \right\|_{E}^{2}=\sum\limits_{k=0}^{n-1}{\left(
n-k \right)}\,Z_{k}^{2}+\sum\limits_{k=1}^{n-1}{k{{\left| r
\right|}^{2}}}Z_{k}^{2}.
\end{equation}
 Here we have two cases depending on $r$.\\

Case 1.If  $\left| r \right|\ge 1$, then from equation (\ref{e4.1}), we
have
$$\left\| {A} \right\|_{E}^{2}\ge \sum\limits_{k=0}^{n-1}{\left( n-k
\right)}Z_{k}^{2}+\sum\limits_{k=1}^{n-1}{k}Z_{k}^{2}=n\sum\limits_{k=0}^{n-1}{Z_{k}^{2}},$$
and from lemma (\ref{l3.2}), we get.
$$\left\| {A} \right\|_{E}^{2}\ge n\left(
Z_{n+1}^{2}-Z_{n-2}^{2}-Z_{n-4}^{2}+T \right)$$

$$\frac{1}{\sqrt{n}}{{\left\| A \right\|}_{E}}\ge \sqrt{\left(
Z_{n+1}^{2}-Z_{n-2}^{2}-Z_{n-4}^{2}+T \right)}$$
By inequality (\ref{e6}), we obtain
\begin{equation}\label{e4.2}
    {{\left\| {A} \right\|}_{2}}\ge \sqrt{\left(
Z_{n+1}^{2}-Z_{n-2}^{2}-Z_{n-4}^{2}+T \right)}.
\end{equation}
On the other hand, let us define two new  matrices ${ C}$ and ${D }$  as :\\\

${C}=\left[ \begin{matrix}
   {{Z}_{0}} & 1 & 1 & \cdots  & 1  \\
   r{{Z}_{n-1}} & r{{Z}_{0}} & 1 & \cdots  & 1  \\
   r{{Z}_{n-2}} & r{{Z}_{n-1}} & r{{Z}_{0}} & \cdots  & 1  \\
   \vdots  & \vdots  & \vdots  & \ddots  & \vdots   \\
   r{{Z}_{1}} & r{{Z}_{2}} & r{{Z}_{3}} & \cdots  & {{Z}_{0}}  \\
\end{matrix} \right]$\, and \,${D}=\left[ \begin{matrix}
   {1} & {{Z}_{1}} & {{Z}_{2}} & \cdots  & {{Z}_{n-1}}  \\
   1 & {{Z}_{0}} & {{Z}_{1}} & \cdots  & {{Z}_{n-2}}  \\
   1 & 1 & {{Z}_{0}} & \cdots  & {{Z}_{n-3}}  \\
   \vdots  & \vdots  & \vdots  & \ddots  & \vdots   \\
   1 & 1 & 1 & \cdots  & {{Z}_{0}}  \\
\end{matrix} \right]$.\\\

Then it is easy to see that  ${A}={C}\circ {D}$
, so from definition (\ref{d.3})
$${{r}_{1}}({C})=\underset{i\le i\le n}{\mathop{\max
}}\,\sqrt{\sum\limits_{j=1}^{n}{{{\left| {{{c}}_{ij}}
\right|}^{2}}}}=\sqrt{Z_{0}^{2}+|r|^{2}\sum\limits_{k=1}^{n-1}{{Z}_{k}^{2}}}$$ and
$${{{c}}_{1}}({D})=\underset{1\le j\le n}{\mathop{\max
}}\,\sqrt{\sum\limits_{i=1}^{n}{{{\left| {{{d}}_{ij}}
\right|}^{2}}}}=\sqrt{1+\sum\limits_{k=1}^{n-1}{{Z}_{k}^{2}}}.$$
Now using theorem (\ref{t.1}), we obtain
 $${{\left\| {A} \right\|}_{2}}\le
{{r}_{1}}({C}){{c}_{1}}({D})=\sqrt{\left( Z_{0}^{2}+{{\left| r \right|}^{2}}\sum\limits_{k=1}^{n-1}{Z_{k}^{2}} \right)\left( 1+\sum\limits_{k=1}^{n-1}{Z_{k}^{2}} \right)}$$
\begin{equation}\label{e4.3}
{{\left\| {A} \right\|}_{2}}\le \sqrt{\left( Z_{0}^{2}+{{\left| r \right|}^{2}}\sum\limits_{k=1}^{n-1}{Z_{k}^{2}} \right)\left( 1+\sum\limits_{k=1}^{n-1}{Z_{k}^{2}} \right)}.
\end{equation}
Combine equations (\ref{e4.2}) and (\ref{e4.3}), we get following inequality
$$\sqrt{\left(
Z_{n+1}^{2}-Z_{n-2}^{2}-Z_{n-4}^{2}+T \right)}\le {{\left\| {A} \right\|}_{2}}\le \sqrt{\left( Z_{0}^{2}+{{\left| r \right|}^{2}}\sum\limits_{k=1}^{n-1}{Z_{k}^{2}} \right)\left( 1+\sum\limits_{k=1}^{n-1}{Z_{k}^{2}} \right)}.$$ Case 2. If
$\left| r \right|\le 1$, then we have
$$\left\| {A} \right\|_{E}^{2}\ge \sum\limits_{k=0}^{n-1}{\left( n-k
\right)}\,{{\left| r
\right|}^{2}}Z_{k}^{2}+\sum\limits_{k=0}^{n-1}{k{{\left| r
\right|}^{2}}}Z_{k}^{2}=n\sum\limits_{k=0}^{n-1}{{{\left| r
\right|}^{2}}Z_{k}^{2}}$$
$$\frac{1}{\sqrt{n}}{{\left\| {A} \right\|}_{E}}\ge |r|\sqrt{Z_{n+1}^{2}-Z_{n-2}^{2}-Z_{n-4}^{2}+T}.$$
By inequality (\ref{e6}), we get \\
\begin{equation}\label{e4.4}
    \left\| {A} \right\|{}_{2}\ge \left| r \right|\sqrt{(Z_{n+1}^{2}-Z_{n-2}^{2}-Z_{n-4}^{2}+T)}.
\end{equation}
On the other hand, let the matrices ${{C}}'$ and ${{D}}'$ be defined as: \\\

${{C}}'=\left[ \begin{matrix}
   {1} & 1 & 1 & \cdots  & 1  \\
   r & {1} & 1 & \cdots  & 1  \\
   r & r & {1} & \cdots  & 1  \\
   \vdots  & \vdots  & \vdots  & \ddots  & \vdots   \\
   r & r & r & \cdots  & {1}  \\
\end{matrix} \right]$  and   ${{D}}'=\left[ \begin{matrix}
   {{{Z}}_{0}} & {{{Z}}_{1}} & {{{Z}}_{2}} & \cdots  & {{{Z}}_{n-1}}  \\
   {{{Z}}_{n-1}} & {{{Z}}_{0}} & {{{Z}}_{1}} & \cdots   & {{{Z}}_{n-2}}  \\
   {{{Z}}_{n-2}} & {{{Z}}_{n-1}} & {{{Z}}_{0}} & \cdots & {{{Z}}_{n-3}}  \\
   \vdots  & \vdots  & \vdots   & \ddots  & \vdots   \\
   {{{Z}}_{1}} & {{{Z}}_{2}} & {{{Z}}_{3}} & \cdots  & {{{Z}}_{0}}  \\
\end{matrix} \right]$\\\\

such that ${A}={{C}}'\circ {{D}}'$, then by definition (1.2), we obtain
$${{r}_{1}}({{C}}')=\underset{1\le i\le n}{\mathop{\max
}}\,\sqrt{\sum\limits_{j=1}^{n}{{{\left| {{{{{c}}'}}_{ij}}
\right|}^{2}}}}=\sqrt{n}=\sqrt{n}$$ and
$${{c}_{1}}({{D}}')=\underset{1\le j\le n}{\mathop{\max
}}\,\sqrt{\sum\limits_{i=1}^{n}{{{\left| {{{{{d}}'}}_{ij}}
\right|}^{2}}}}=\sqrt{\sum\limits_{k=0}^{n-1}{{Z}_{k}^{2}}}=\sqrt{\left(
Z_{n+1}^{2}-Z_{n-2}^{2}-Z_{n-4}^{2}+T \right)}.$$
Again by applying theorem (\ref{t.1}), we get\\
$${{\left\| {A} \right\|}_{2}}\le
{{r}_{1}}({{C}}'){{c}_{1}}({{D}}')=\sqrt{\left(
Z_{n+1}^{2}-Z_{n-2}^{2}-Z_{n-4}^{2}+T
\right)}\sqrt{n},$$
\begin{equation}\label{e4.5}
{{\left\| {A} \right\|}_{2}}\le \sqrt{n\left(
Z_{n+1}^{2}-Z_{n-2}^{2}-Z_{n-4}^{2}+T \right)},
\end{equation}
and combing inequality (\ref{e4.4}) and (\ref{e4.5}), we obtain the required result. \\
$$\left| r \right|\sqrt{\left(
Z_{n+1}^{2}-Z_{n-2}^{2}-Z_{n-4}^{2}+T \right)}\le
{{\left\| {A} \right\|}_{2}}\le \sqrt{n\left(
Z_{n+1}^{2}-Z_{n-2}^{2}-Z_{n-4}^{2}+T \right)}.$$
\end{proof}

\begin{thm}
Let ${A}$ be the circulant matrix on generalized pell padovan
sequence. Then,  $${{\left\|{ A}
\right\|}_{E}}=\sqrt{n\sum\limits_{i=0}^{n-1}{Z_{i}^{2}}}=\sqrt
n\sqrt{Z_{n+1}^{2}-Z_{n-2}^{2}-Z_{n-4}^{2}+T}.$$ and
 $$\sqrt{Z_{n+1}^{2}-Z_{n-2}^{2}-Z_{n-4}^{2}+T}
\le{{\left\|{ A} \right\|}_{2}}\le \sqrt{\left( \sum\limits_{i=0}^{n-1}{Z_{i}^{2}} \right)\left( 1+\sum\limits_{i=1}^{n-1}{Z_{i}^{2}} \right)}$$\\
\end{thm}
\begin{proof}
Since by definition of circulant Matrix, the matrix ${U}$ is of the form  \\
$${A}=\left[ \begin{matrix}
   {Z_{0}} & {Z_{1}} & {Z_{2}} & \cdots  & {Z_{n-1}}  \\
   {Z_{n-1}} & {Z_{0}} & {Z_{1}} & \cdots   & {Z_{n-2}}  \\
   {Z_{n-2}} & {Z_{n-1}} & {Z_{0}} & \cdots    & {Z_{n-3}}  \\
     \vdots  & \vdots  & \vdots  & \ddots  & \vdots  &   \\
   {Z_{1}} & {Z_{2}} & {Z_{3}} & \cdots    & {Z_{0}}  \\
\end{matrix} \right].$$
By definition of Euclidean norm, we have \\
\begin{equation}
    {{\left\|{ A}
\right\|}_{E}}=\sqrt{n\sum\limits_{i=0}^{n-1}{Z_{i}^{2}}}=\sqrt
n\sqrt{Z_{n+1}^{2}-Z_{n-2}^{2}-Z_{n-4}^{2}+T}.
\end{equation}
By inequality (\ref{e6}), we get
\begin{equation}\label{e4.7}
 \sqrt{Z_{n+1}^{2}-Z_{n-2}^{2}-Z_{n-4}^{2}+T}\le {{\left\| {A} \right\|}_{2}}.
\end{equation}
Let matrices ${B}$ and ${C}$ be defined as:\\

$ {B} = \left\{
        \begin{array}{ll}
            {{{b}}_{ij}}={{Z}_{\left( \,\bmod \,\left( j-i,n \right) \right)}} & \quad i\leq j \\
            {{{b}}_{ij}}=1 & \quad i < j
        \end{array}
    \right.
$ and $ {C} = \left\{
        \begin{array}{ll}
            {{{c}}_{ij}}={{Z}_{\left( \,\bmod \,\left( j-i,n \right) \right)}} & \quad i < j \\
            {{{c}}_{ij}}=1 & \quad i \geq j
        \end{array}
    \right.
$\\

Then the row norm and column norm of $ {B}$ and $ {C}$
are given as:\\
$${{r}_{1}}({B})=\underset{i}{\mathop{\max
}}\,\sqrt{\sum\limits_{j=1}^{n}{\left| {{{b}}_{ij}}
\right|{}^{2}}}=\sqrt{\sum\limits_{i=0}^{n-1}{Z_{i}^{2}}}$$
and
$${{c}_{1}}({C})=\underset{j}{\mathop{\max
}}\,\sqrt{\sum\limits_{i=1}^{n}{\left| {{{c}}_{ij}}
\right|{}^{2}}}=\sqrt{1+\sum\limits_{i=1}^{n-1}{Z_{i}^{2}}}.$$
Using theorem  (\ref{t.1}), we have \\
\begin{equation}\label{e4.8}
  {{\left\| {A} \right\|}_{2}}\le \sqrt{\left( \sum\limits_{i=0}^{n-1}{Z_{i}^{2}} \right)\left( 1+\sum\limits_{i=1}^{n-1}{Z_{i}^{2}} \right)}.
\end{equation}
Combine (\ref{e4.7}) and (\ref{e4.8}), we get \\
$$\sqrt{Z_{n+1}^{2}-Z_{n-2}^{2}-Z_{n-4}^{2}+T}
\le{{\left\|{ A} \right\|}_{2}}\le \sqrt{\left( \sum\limits_{i=0}^{n-1}{Z_{i}^{2}} \right)\left( 1+\sum\limits_{i=1}^{n-1}{Z_{i}^{2}} \right)}.$$
\end{proof}

\begin{thm}
Let $A$ be an $n\times n$ semi-circulant matrix ${A}=({a}_{ij})$
with the generalized Pell Padovan  sequence. Then,
$$\left\| {A} \right\|_{E}^{2}=Z_{n+1}^{2}+2Z_{n}^{2}+2Z_{n-1}^{2}+Z_{n-2}^{2}+Z_{n-3}^{2}-Z_{-2}^{2}-Z_{-1}^{2}-{{b}^{2}}-{{c}^{2}}-2{{a}^{2}}+\left( n-1 \right)\left( T-{{a}^{2}} \right).$$
\end{thm}

\begin{proof}
For the semi-circulant matrix ${A}=({a}_{ij})$ with
the Generalized tribonacci sequence
numbers we have \\
$${{{a}}_{ij}}=\left\{ \begin{array}{*{35}{l}}
   {{R}_{j-i+1}} & \quad i\le j  \\
   0 & \quad otherwise . \\
\end{array} \right.$$
From the definition of Euclidean norm, we have \\
$${{\left\| {A}
\right\|}_{E}^2}=\sum\limits_{j=1}^{n}{\sum\limits_{i=1}^{j}{{{\left(
{{R}_{j-i+1}} \right)}^{2}}}}=
 \sum\limits_{j=1}^{n}{\left( \sum\limits_{k=1}^{j}{R_{k}^{2}}
 \right)}.$$\\
Using lemma (\ref{l4.1}), we get the required result
 $$\left\| {A} \right\|_{E}^{2}=Z_{n+1}^{2}+2Z_{n}^{2}+2Z_{n-1}^{2}+Z_{n-2}^{2}+Z_{n-3}^{2}-Z_{-2}^{2}-Z_{-1}^{2}-{{b}^{2}}-{{c}^{2}}-2{{a}^{2}}+\left( n-1 \right)\left( T-{{a}^{2}} \right).$$

\end{proof}
\begin{thm}\label{et4.10}
 If ${A}= ({{{a}}_{ij}} )$ is an $n\times n$ Hankel matrix with
${{{a}}_{ij}}={{Z}_{i+j-1}}$,
then \small
$${{\left\| \text{ }A \right\|}_{E}}={{\left( Z_{2n}^{2}+2Z_{2n-1}^{2}+2Z_{2n-2}^{2}+Z_{2n-3}^{2}+Z_{2n-4}^{2}-2Z_{n}^{2}-4Z_{n-1}^{2}-4Z_{n-2}^{2}-2Z_{n-3}^{2}-2Z_{n-4}^{2}+S \right)}^{\frac{1}{2}}},$$
\normalsize
where $S=Z_{-2}^{2}+Z_{-1}^{2}+{{b}^{2}}+{{c}^{2}}+2{{a}^{2}}+2\left( T-{{a}^{2}} \right)$.
\end{thm}
\begin{proof}
From the definition of Hankel matrix, the matrix ${A}$ is of the form  \\
$$
{A}=\left[ \begin{matrix}
   {{Z}_{1}} & {{Z}_{2}} & {{Z}_{3}} & \cdots  & {{Z}_{n-1}} & {{Z}_{n}}  \\
   {{Z}_{2}} & {{Z}_{3}} & {{Z}_{4}} & \cdots  & {{Z}_{n}} & {{Z}_{n+1}}  \\
   {{Z}_{3}} & {{Z}_{4}} & {{Z}_{5}} & \cdots  & {{Z}_{n+1}} & {{Z}_{n+2}}  \\
   \vdots  & \vdots  & \vdots  & \ddots  & \vdots  & \vdots   \\
   {{Z}_{n-1}} & {{Z}_{n}} & {{Z}_{n+1}} & \cdots  & {{Z}_{2n-3}} & {{Z}_{2n-2}}  \\
   {{Z}_{n}} & {{Z}_{n+1}} & {{Z}_{n+2}} & \cdots  & {{Z}_{2n-2}} & {{Z}_{2n-1}}  \\
\end{matrix} \right].$$

So, we have
 $${{\left\| {A} \right\|}_{E}}={{\left(
\sum\limits_{i=1}^{m}{\sum\limits_{j=1}^{n}{{{\left|
{{{a}}_{ij}} \right|}^{2}}}} \right)}^{{}^{1}/{}_{2}}}.$$

$${{\left\|{ A} \right\|}_{E}}={{\left(
\sum\limits_{k=1}^{n}{{Z}_{k}^{2}+\sum\limits_{k=2}^{n+1}{{Z}_{k}^{2}}+...+\sum\limits_{k=n}^{2n-1}{{Z}_{k}^{2}}}
\right)}^{{}^{1}/{}_{2}}}.$$

$${{\left\| {A} \right\|}_{E}}={{\left( \left(
\sum\limits_{k=1}^{n}{{Z}_{k}^{2}+\sum\limits_{k=1}^{n+1}{{Z}_{k}^{2}}+...+\sum\limits_{k=1}^{2n-1}{{Z}_{k}^{2}}}
\right)-\left(
\sum\limits_{k=1}^{n-1}{\sum\limits_{i=1}^{k}{{Z}_{i}^{2}}} \right)
\right)}^{{}^{1}/{}_{2}}}.$$

$${{\left\| \text{ }A \right\|}_{E}}={{\left( \left( {{Q}_{n}} \right)+({{Q}_{n+1}})+...+\left( {{Q}_{2n-1}} \right)-\sum\limits_{k=1}^{n-1}{\left( {{Q}_{n}} \right)} \right)}^{\frac{1}{2}}}$$
$${{\left\| \text{ }A \right\|}_{E}}={{\left( \sum\limits_{k=n}^{2n-1}{\left( {{Q}_{k}} \right)}-\sum\limits_{k=1}^{n-1}{\left( {{Q}_{n}} \right)} \right)}^{\frac{1}{2}}}.$$
$${{\left\| \text{ }A \right\|}_{E}}={{\left( \sum\limits_{k=1}^{2n-1}{\left( {{Q}_{k}} \right)}-2\sum\limits_{k=1}^{n-1}{\left( {{Q}_{k}} \right)} \right)}^{\frac{1}{2}}}.$$
$${{\left\| \text{ }A \right\|}_{E}}={{\left( \sum\limits_{k=1}^{2n-1}{\sum\limits_{i=1}^{k}{R_{i}^{2}}}-2\sum\limits_{k=1}^{n-1}{\sum\limits_{i=1}^{k}{R_{i}^{2}}} \right)}^{\frac{1}{2}}}.$$
\small
$${{\left\| \text{ }A \right\|}_{E}}={{\left( Z_{2n}^{2}+2Z_{2n-1}^{2}+2Z_{2n-2}^{2}+Z_{2n-3}^{2}+Z_{2n-4}^{2}-2Z_{n}^{2}-4Z_{n-1}^{2}-4Z_{n-2}^{2}-2Z_{n-3}^{2}-2Z_{n-4}^{2}+S \right)}^{\frac{1}{2}}}.$$
\normalsize

\end{proof}

\begin{thm}
If ${A}= ({{{a}}_{ij}} )$ is an $n\times n$ Hankel matrix with
${{{a}}_{ij}}={{R}_{i+j-1}}$,then,
$$\frac{1}{\sqrt{n}}{{\left\| {A} \right\|}_{E}}\leq{{\left\| {A} \right\|}_{2}}\le \sqrt{\left( 1+Z_{n+2}^{2}-Z_{n-1}^{2}-Z_{n-3}^{2}+T-Z_{0}^{2}-Z_{1}^{2} \right)(Z_{n+2}^{2}-Z_{n-1}^{2}-Z_{n-3}^{2}+T-Z_{0}^{2}}).$$
\end{thm}
\begin{proof}
From theorem (\ref{et4.10}) and inequality (\ref{e6}), we have
$$\frac{1}{\sqrt{n}}{{\left\| {A} \right\|}_{E}}\le {{\left\| {A}
\right\|}_{2}}$$

Let us define two new matrices,\\\

 ${{M}}= \left\{
        \begin{array}{ll}
          {m}_{ij}={{Z}_{i+j-1}} & i\leq j\\
          {n}_{ij}=1 & i> j\,
          \,\,
        \end{array}
    \right.
$ and\, ${{N}}= \left\{
        \begin{array}{ll}
          {n}_{ij}={{Z}_{i+j-1}} & i> j\,\\
          {n}_{ij}=1 & i\leq j,
          \,\,
        \end{array}
    \right.
$\\

It can be  easily seen that ${A}={{M}}\circ
{{N}}$. Thus we get\\

$${{r}_{1}}\left( {{M}} \right)=\underset{i}{\mathop{\max
}}\,\sqrt{\sum\limits_{j}{{{\left| {{{m}}_{ij}}
\right|}^{2}}}}=\sqrt{\sum\limits_{i=1}^{n}{Z_{i}^{2}}}=\sqrt{Z_{n+2}^{2}-Z_{n-1}^{2}-Z_{n-3}^{2}+T-Z_{0}^{2}}$$\\
and\\
$${{c}_{1}}\left( {{N}} \right)=\underset{j}{\mathop{\max
}}\,\sqrt{\sum\limits_{i}{{{\left| {{{n}}_{ij}}
\right|}^{2}}}}=\sqrt{1+\sum\limits_{i=2}^{n}{Z_{i}^{2}}}=\sqrt{1+Z_{n+2}^{2}-Z_{n-1}^{2}-Z_{n-3}^{2}+T-Z_{0}^{2}-Z_{1}^{2}}.$$

Using the theorem (\ref{t.1}), we have
$${{\left\| {A} \right\|}_{2}}\le \sqrt{\left( 1+Z_{n+2}^{2}-Z_{n-1}^{2}-Z_{n-3}^{2}+T-Z_{0}^{2}-Z_{1}^{2} \right)(Z_{n+2}^{2}-Z_{n-1}^{2}-Z_{n-3}^{2}+T-Z_{0}^{2}}).$$
\end{proof}
\begin{thm}
If ${A}= ({{{a}}_{ij}} )$ is an $n\times n$ Hankel
matrix with
${{{a}}_{ij}}={{Z}_{i+j-1}}$,\\
then, $${{\left\| {A} \right\|}_{1}}={{\left\| {A} \right\|}_{\infty
}}={{Z}_{2n+4}}-{{Z}_{n+4}}.$$
\end{thm}

\begin{proof}
From the definition of the matrix ${A}$ , we can write\\
$${{\left\| {A} \right\|}_{1}}=\underset{i\le j\le n}{\mathop{\max
}}\,\sum\limits_{i=1}^{n}{\left| {{{a}}_{ij}}
\right|}=\underset{1\le j\le n}{\mathop{\max }}\,\left\{ \left|
{{{a}}_{1j}} \right|+\left| {{{a}}_{2j}}
\right|+\left| {{{a}}_{3j}} \right|...\left|
{{{a}}_{nj}} \right| \right\}$$
$${{\left\|{ A}
\right\|}_{1}}={{Z}_{n}}+{{Z}_{n+1}}+{{Z}_{n+2}}+\cdots+{{Z}_{2n-1}}$$
$${{\left\| {A}
\right\|}_{1}}=\sum\limits_{i=1}^{2n-1}{{{Z}_{i}}}-\sum\limits_{i=1}^{n-1}{{{Z}_{i}}},$$
by lemma (\ref{l3.1}), we have \\
$${{\left\| {A}
\right\|}_{1}}={{Z}_{2n+4}}-{{Z}_{n+4}}.$$
Similarly the row norm of the matrix ${A}$ can be computed as: \\\
$${{\left\|{ A} \right\|}_{\infty }}=\underset{1\le i\le
n}{\mathop{\max }}\,\sum\limits_{j=1}^{n}{\left| {{{a}}_{ij}}
\right|}={{Z}_{2n+4}}-{{Z}_{n+4}}.$$
\end{proof}


\begin{thebibliography}{00}
\bibitem {Akbulak} Akbulak.M , D. Bozkurt, On the norms of Toeplitz matrices involving
Fibonacci and Lucas numbers, Hacettepe journal of Mathematics and
Statistics 37(2), 89-95, (2008)

\bibitem {Arzu}Coskun. A. T, On the some properties of circulant matrix with third order linear recurrent sequence (arXive:1406.5349v1)
(2014)
\bibitem {Halici}Halici. S, On some inequality and hankel matrices
involving Pell,Pell Lucas numbers, math.reports 15(65), 1-10,
1(2013)
\bibitem {Kalman}Kalman. D, R. Mena, The Fibonacci Numbers -
Exposed, The Mathematical Magazine, 2(2002).

\bibitem {Koshy}Koshy. T, Fibonacci and Lucas Numbers with Applications (Wiley-Interscience Publications,
(2001).

.
\bibitem {Kilic}Kilic. E, Sums of the squares of terms of sequence {un}, Indian Acad.
Sci.(Math.Sci.), 118(1), 27-41, (2008)

\bibitem {Kocer}Kocer. E. G,Circulant, Negacyclic and Semicirculant matrices with the modified
Pell, Jacobsthal and Jacobsthal- Lucas numbers, Hacettepe Journal of
Mathematics and Statistics 36(2), 133-142, (2007)


\bibitem {Mathias}Mathias. R, The spectral norm of nonnegative matrix, Linear Algebra and its Applications
131, 269-284, (1990).

\bibitem{Ali}Raza. Z, M.A. Ali, on the norms of
somr special matrices with pell padovan and pell padovan like sequence (Pre Print)

\bibitem {Reams}Reams. R, ''Hadamard inverses , square roots and products
of almost semidefinite matrices'', Linear Algebra Appl.
288,35-43,(1999).

\bibitem {Solak}Solak. S, On the norms of circulant matrices with the Fibonacci and Lucas numbers, Appl. Math. Comput., 160, 125-132, (2005).

\bibitem {Shen}Shen. S, J. Cen, On the spectral norms of r -circulant matrices with
the k-Fibonacci and k-Lucas numbers, Int. J. Contemp. Math.
Sciences, 5(12), 569-578, (2010).

\bibitem {Bozkurt}Solak. S, D. Bozkurt, On the spectral norms of Cauchy-Toeplitz and
Cauchy-Hankel matrices, Appl. Math. Comput., 140, 231-238, (2003).

\bibitem {Yazlik Y}Yazlik Y., Taskara N., spectral norm, eigen values and determinant of circulant matrix
involving the generalized 4$k$-Horadom numbers, Ars Combinatoria, 204, 505-512, (2012).
\bibitem {Zielke}Zielke. G, Some remarks on matrix norms, condition numbers and error estimates for linear equations. Linear Algebra Appl. 110,
29-41,(1988)




\
\end{thebibliography}
\end{document}